\documentclass{amsart}
\usepackage{graphicx}
\usepackage{dsfont}
\usepackage{amsmath,amssymb}

\def\df{\mathrm{d}}
\def\im{\mathrm{i}}
\def\ie{i.e.\ }
\def\eg{e.g.\ }
\def\matlab{\textsc{Matlab}}
\newcommand{\boldchr}[1]{\mbox{\boldmath $#1$}}

\newcommand{\btop}[2]{\genfrac{}{}{0pt}{3}{#1}{#2}}
\newtheorem{theorem}{Theorem}[subsection]
 \newtheorem{corollary}[theorem]{Corollary}
 \newtheorem{lemma}[theorem]{Lemma}
 
 \theoremstyle{definition}
 
 \theoremstyle{remark}
 \newtheorem{remark}[theorem]{Remark}

\begin{document}
\title[H. Majidian: MODIFIED FILON-CLENSHAW-CURTIS RULES] {Modified Filon-Clenshaw-Curtis rules for oscillatory integrals with a nonlinear oscillator}
\author{Hassan Majidian$^1$}

\thanks{$^1$Department of Multidisciplinary Studies, Faculty of Encyclopedia Studies, Institute for Humanities and Cultural Studies, PO Box 14155-6419, Tehran, Iran,
\\ \indent\,\,\,e-mail: h.majidian@ihcs.ac.ir}

\begin{abstract}
Filon-Clenshaw-Curtis (FCC) rules are among rapid and accurate quadrature rules for computing oscillatory integrals. In the implementation of the FCC rules, when the oscillator of the integral is nonlinear, its inverse should be evaluated at several points. In this paper, we suggest an approach based on interpolation, which leads to a class of modifications of the original FCC rules in such a way that the modified rules do not deal with the inverse of the oscillator function. In the absence of stationary points, two reliable and efficient algorithms based on modified FCC (MFCC) rules are introduced. For each algorithm, an error estimate is given theoretically, and then illustrated by some numerical experiments. Also, some numerical experiments are carried out in order to compare convergence speed of the two algorithms. In the presence of stationary points, an algorithm based on composite MFCC rules on graded meshes is developed. An error estimate is given theoretically, and then illustrated by some numerical experiments.


\end{abstract}
\maketitle

\smallskip
\section{Introduction}\label{sec:intro}
Consider integrals of the form
\begin{equation}\label{eq:intfg}
I^{[a,b]}_k(f,g):=\int_{a}^b f(x)\exp(\im k g(x))\,\df x,
\end{equation}
where ${f\in L^1[a,b]}$, ${k>0}$, and ${g\in C^m[a,b]}$ for some positive integer $m$. If $g$ does not oscillate rapidly in ${[a,b]}$, the integrand in~\eqref{eq:intfg} oscillates violently for larger values of $k$. This class of integrals contains a large portion of highly oscillatory integrals, appearing in many areas of science and engineering, e.g., Fourier series and transforms, special functions, high-frequency acoustic scattering, etc (see, \eg\cite{olv08,chan12} and references therein). Because of their wide applications, computing oscillatory integrals of the form~\eqref{eq:intfg} has been the subject of many researches in the last two decades.

Filon-type methods are among most efficient ones for computing~\eqref{eq:intfg} with a long history. While high accurate methods based on steepest descents (see, \eg\cite{huy06,dea09,huy12}) need some manual calculations regarding the steepest decent paths in the complex plane, the Filon-type methods does not deals with the complex calculus and can be performed automatically by computers provided that the moments can be computed to the desired accuracy. The idea, in general, is to replace the amplitude function $f$ by an interpolation polynomial $p_n$ and consider ${I^{[a,b]}_k(p_n,g)}$ as an approximation to ${I^{[a,b]}_k(f,g)}$. This idea originated from~\cite{fil28}, where Louis Napoleon George Filon (1875--1937) applied it to the Fourier integral ${\int_a^{b}f(x)\sin k x\,\df x}$.

After the work of Filon, many papers on his method appeared, but the 2005 paper~\cite{ise05} proved to be a milestone in the history of research on the Filon-type methods. In that work, the asymptotic expansion of ${I^{[a,b]}_k(f,g)}$ by negative powers of $k$ has been studied, and a generalization of the Filon's method has been developed. The paper has been the motivation of many later works.

Successful implementation of a Filon-type method rests on the ability to compute the moments ${I^{[a,b]}_k(x^m,g)}$. For the linear oscillator ${g(x)=x}$, the moments can be computed by the following identity:
\begin{equation}\label{eq:momx}
I^{[a,b]}_k(x^m,x)=\frac{1}{(-\im k)^{m+1}}\left[\Gamma(1+m,-\im k a)-\Gamma(1+m,-\im k b)\right],
\end{equation}
where $\Gamma$ is the incomplete Gamma function~\cite{abr65}. However, formula~\eqref{eq:momx} is useless mainly because interpolation by the polynomial basis ${1,x,x^2,\ldots}$ is in general unstable. For a more complex oscillator $g$, Olver~\cite{olv07} considered a $g$-dependent basis instead of the usual polynomial basis $x^j$. Then the moments can be written in closed forms, again in terms of incomplete Gamma functions. Another approach, proposed in~\cite{xia07}, is based on the transformation ${\tau=g(x)}$. The method enables one to compute the moments by the identity~\eqref{eq:momx} when $g$ has no stationary points in ${[a,b]}$; otherwise if $g$ has a single stationary point, similar identities for ${I^{[0,y]}_k(x^m,x^j)}$ are employed (cf.~\cite{abr65,ise05}).

Filon-Clenshaw-Curtis (FCC) rules, beside other advantages, enable one to compute the moments stably and rapidly without needing to deal with (incomplete) Gamma functions, whose evaluation is not trivial. The (${N+1}$)-point FCC rule can be described as follows: Consider the oscillatory integral
\begin{equation}\label{eq:intf}
I_k(f):=\int_{-1}^1 f(x)\exp(\im k x)\,\df x.
\end{equation}
In the (${N+1}$)-point FCC rule, the amplitude function $f$ is interpolated at Clenshaw-Curtis points ${t_{j,N}:=\cos(j\pi/N)}$, ${j=0,\ldots,N}$ by the polynomial
\begin{equation}
Q_Nf(s):=\sum_{n=0}^N{}''\alpha_{n,N}(f)T_n(s),
\end{equation}
where ${T_n(s)=\cos(n\arccos(s))}$ are the Chebyshev polynomials of the first kind, and $\sum{}''$ means that the first and the last terms of the sum are to be halved. By the discrete orthogonality of the cosine function, the coefficients $\alpha_{n,N}(f)$ can be written as
\begin{equation}\label{eq:alpha}
\alpha_{n,N}(f)=\frac{2}{N}\sum_{j=0}^N{}''\cos\left(jn\pi/N\right)f(t_{j,N}),\quad n=0,\ldots,N.
\end{equation}
If $f$ in~\eqref{eq:intf} is replaced by $Q_Nf$, then the (${N+1}$)-point FCC rule is obtained as
\begin{equation}\label{eq:fcc}
I_{k,N}(f)=\sum_{n=0}^N{}''\alpha_{n,N}(f)\omega_n(k),
\end{equation}
where the weights (modified moments) ${\omega_n(k)=I_k(T_n)}$, ${n\ge0}$, can be computed recursively thanks to the three-term recurrence relation for the Chebyshev polynomials $T_n$. As a generalization of the original Clenshaw-Curtis rules~\cite{cle60}, this idea has been developed gradually by employing the results of many earlier works, mainly Piessens' contributions~\cite{pie71,pie73,pie76,pie84,pie00} and Sloan's works~\cite{slo78,slo80}. By a simple affine-like
change of variables, the FCC rules~\eqref{eq:fcc} can also be applied to
\begin{equation}\label{eq:intfab}
I^{[a,b]}_k(f):=\int_{a}^b f(x)\exp(\im k x)\,\df x.
\end{equation}

The recent papers of Dom\'{\i}nguez et al. \cite{dom11,dom13} contain important developments and results on the FCC rules. In~\cite{dom11}, beside an error bound in terms of both $k$ and $N$, a stable algorithm with the complexity of ${\mathcal{O}(N\log N)}$ was proposed for computing the weights. The algorithm can be employed for computing the general integral~\eqref{eq:intfg} in the following way (cf.~\cite{dom13}): If $g$ has no stationary points in ${[a,b]}$, then we can assume that ${g'(x)>0}$, ${x\in[a,b]}$, without loss of generality. In this case, the change of variables ${\tau=g(x)}$ reduces the integral to ${I^{[g(a),g(b)]}_k(F)}$ with ${F(\tau)=\left((f/g')\circ g^{-1}\right)(\tau)}$, ${\tau\in[g(a),g(b)]}$; otherwise, if $g$ has some stationary points in ${[a,b]}$, similar change of variables reduce the integral to a finite number of integrals of the form~\eqref{eq:intfab}, such that each integral has a single singularity at one of the endpoints and can be computed efficiently by composite FCC rules on graded meshes~\cite{dom13}.

\paragraph{Statement of the problem (or the motivation to this paper)} In the implementation of the algorithm proposed in~\cite{dom13}, computing $g^{-1}$ is required at several points. Expressing $g^{-1}(x)$ in a closed form, however, is not always possible. For example, consider the function ${g(x)=x-\sin(x)}$. Thus, we should resort to approximation methods, e.g., Newton iterations. This however imposes extra error and computational cost.

In this paper, we follow the idea of~\cite{xia07} and propose reliable and efficient methods for computing~\eqref{eq:intfg}, which are based on the FCC rules but without needing to compute $g^{-1}(x)$ at any point. The modified FCC rules, as developed in this paper, play key roles when exponential transformations are applied to the oscillatory integrals in order to treat singularities (see~\cite{maj19} for details).

Authors in~\cite{ma18} have also employed the ideas in~\cite{xia07} and suggest composite Filon-type rules which do not deal with $g^{-1}$, but they do not benefit from advantages of the FCC rule. Our idea in this paper is however essentially different in that it is based on the FCC rules and exploits their fast and stable construction (cf.~\cite{dom11}) as well as their error estimates in terms of $k$, $N$, and Sobolev regularity of $f$ (cf.~\cite{dom13}). These advantages are gained by employing special interpolations such that the $g$-dependent interpolation nodes of~\cite{xia07} approximately coincide the Clenshaw-Curtis points.

The plan of this paper is as follows: In Section~\ref{sec:gen}, which contains the main ideas, we describe a class of quadrature rules for computing~\eqref{eq:intfg} in the absence of stationary points (i.e., $g(x)\ne0$, for any $x\in[a,b]$). The rules can be considered as modifications of the FCC rules~\cite{dom13} which release the necessity of computing $g^{-1}(x)$ by imposing an interpolation process. Thus, an interpolation error is added to the total error. In order to decrease the interpolation error, two efficient interpolation methods are employed, based on which two algorithms are proposed. The next two sections deal with these algorithms, their error estimates, and some numerical experiments. In Section~\ref{sec:comp}, we compare convergence speed of the two algorithms. In Section~\ref{sec:statpt}, the case when $g$ has a finite number of stationary points in $[a,b]$ is studied; we show how the modified FCC rules can be applied in this case. An error estimate is provided, and some numerical results are given. Finally, we bring a conclusion.
\paragraph{Some general remarks} Throughout the paper, $C$ and $C'$ stand for generic constants independent of $k$, and their values may differ from place to place. Their dependency on other parameters (if exist) is declared in each section. Note also that the space of $m$ times continuously differentiable functions is denoted by the common notation $C^m[a,b]$. Also, in all of the numerical experiments, the reference values of the sample integrals have been computed by Mathematica to at least 10 digits higher than the machine precision. All the numerical experiments are performed in \matlab\ R2015b by my personal laptop.\footnote{Intel Core i7-8550 CPU with a clock speed between 1.80 GHz and 1.99 GHz with 8GB of RAM.}
\section{The general idea}\label{sec:gen}
Consider the integral~\eqref{eq:intfg} with $g$ having no stationary points in $[a,b]$. Then, we can assume that $g'(x)>0$, $x\in[a,b]$, without loss of generality. By this assumption, the integral is changed into an integral with the linear oscillator. Indeed, by the change of variables $\tau=g(x)$,
\begin{align}
I^{[a,b]}_k(f,g)&:=\int_{g(a)}^{g(b)} F(\tau)\exp(\im k \tau)\,\df\tau \nonumber\\[1ex]
&=l\exp(\im k c)I_{\tilde{k}}(\widetilde{F}) \nonumber\\[1ex]
&=l\exp(\im k c)\int_{-1}^1 \widetilde{F}(\tau)\exp(\im\tilde{k}\tau)\,\df\tau,\label{eq:intftilde}
\end{align}
where
\begin{align}
F(\tau)&=\left((f/g')\circ g^{-1}\right)(\tau),\quad\tau\in[g(a),g(b)] \label{eq:F}\\[1ex]
c&=\frac{g(b)+g(a)}{2},\quad l=\frac{g(b)-g(a)}{2}, \\[1ex]
\tilde{k}&=lk,\quad \widetilde{F}(\tau)=F(c+l\tau).
\end{align}

The efficient and robust algorithm of~\cite{dom11} for constructing the FCC rules can now be applied to~\eqref{eq:intftilde}. However, evaluating $g^{-1}(x)$ for some $x$ is required in this process, and this is not of our favorite. For example, it may be impossible to express $g^{-1}(x)$ in a closed form. Then we should resort to Newton iterations, while their convergence is not guaranteed. Even if it converges rapidly, the idea imposes extra error and computational cost.

In the following, we develop a modification of the method in such a way that the need for computing $g^{-1}(x)$ is released. The idea can then be employed in the presence of stationary points by means of the composite rules on graded meshes, as developed in~\cite{dom13}. This idea is discussed in Section~\ref{sec:statpt}.
\subsection{The modified Filon-Clenshaw-Curtis (MFCC) rule}\label{sub:mfcc}
Employing the (${N+1}$)-point FCC rule~\eqref{eq:fcc} for computing $I_{\tilde{k}}(\widetilde{F})$ (that is essential  in~\eqref{eq:intftilde}) necessitates evaluation of $\widetilde{F}$ at the Clenshaw-Curtis points. We describe a method for accurate approximation of $\widetilde{F}$ at the Clenshaw-Curtis points, in which $g^{-1}$ is not needed to be evaluated at any point.

Let $N$ be a positive integer and consider $N+1$ Clenshaw-Curtis points $t_{0,N},\ldots,t_{N,N}$. Let $\ell_1:[-1,1]\to[a,b]$ and $\ell_2:[-1,1]\to[g(a),g(b)]$ be the onto affine-like mappings. By definition, $\widetilde{F}=F\circ\ell_2$. For a large integer $N'>1$, choose a set of arbitrary points ${-1=u_0<\cdots<u_{N'}=1}$ and define
\begin{equation}\label{eq:d}
d_j:=\ell^{-1}_2(g(\ell_1(u_j))),\quad j=0,\ldots,N'.
\end{equation}
Then $-1=d_0<\cdots<d_{N'}=1$. One can compute $\widetilde{F}(d_j)$ without needing to compute $d_j$ itself and to evaluate the inverse of $g$ at any point. Indeed,
\begin{equation}\label{eq:Fd}
\widetilde{F}(d_j)=F(\ell_2(d_j))=\dfrac{f(\ell_1(u_j))}{g'(\ell_1(u_j))}.
\end{equation}

Now, we can approximate $\widetilde{F}(t_{j,N})$ by interpolation of $\widetilde{F}$ at the nodes ${d_0,\ldots,d_{N'}}$. Denote by $\widetilde{Q}_{N'}$ the interpolating projection corresponding to some method of interpolation at the points ${-1=d_0<\cdots<d_{N'}=1}$, i.e.,
\begin{equation}
\widetilde{Q}_{N'}\widetilde{F}(d_j)=\frac{f(\ell_1(u_j))}{g'(\ell_1(u_j))},\quad j=0,\ldots,N'.
\end{equation}
Then,
\begin{equation}
\widetilde{Q}_{N'}\widetilde{F}(t_{j,N})\approx\widetilde{F}(t_{j,N}),\quad j=0,\ldots,N.
\end{equation}

If $\tilde{k}\ge1/2$, \ie $k\ge1/(g(b)-g(a))$, the (${N+1}$)-point MFCC rule is defined by
\begin{equation}\label{eq:mfcc}
I^{[a,b]}_{k,N}(f,g):=l\exp(\im kc)I_{\tilde{k},N}(\widetilde{Q}_{N'}\widetilde{F})=l\exp(\im kc)\sum_{n=0}^N{}''\alpha_{n,N}(\widetilde{Q}_{N'}\widetilde{F})\omega_n(\tilde{k}),
\end{equation}
where the coefficients are expressed by ~\eqref{eq:alpha}, i.e.,
\begin{equation}
\alpha_{n,N}(\widetilde{Q}_{N'}\widetilde{F})=\frac{2}{N}\sum_{j=0}^N{}''\cos\left(jn\pi/N\right)\widetilde{Q}_{N'}\widetilde{F}(t_{j,N}),\quad n=0,\ldots,N.
\end{equation}
Indeed, the vector ${(\alpha_{n,N}(\widetilde{Q}_{N'}\widetilde{F}))_{n=0}^N}$ is the discrete cosine transform (DCT) of type I of \sloppy ${(\widetilde{Q}_{N'}\widetilde{F}(t_{j,N}))_{j=0}^N}$, which can be computed by the FFT algorithm in $\mathcal{O}(N\ln N)$ flops. Also, the weights $\omega_n(\tilde{k})$ are stably and rapidly computed by the two-phase algorithm~\cite{dom11}.

In summary, the MFCC rule for computing~\eqref{eq:intfg} is as follows:
\begin{itemize}
\item
If $\tilde{k}\ge1/2$,
$$
I^{[a,b]}_k(f,g)=l\exp(\im k c)I_{\tilde{k}}(\widetilde{F})\approx l\exp(\im k c)I_{\tilde{k},N}(\widetilde{F})\approx l\exp(\im k c)I_{\tilde{k},N}(\widetilde{Q}_{N'}\widetilde{F}).
$$
\item
If $\tilde{k}<1/2$, the integral $I_{\tilde{k}}(\widetilde{F})$ and then $I^{[a,b]}_k(f,g)$ is no longer oscillatory, and can be computed by the standard Clenshaw-Curtis rule:
$$
I^{[a,b]}_{k,N}(f,g):=I_{0,N}(H_k),\quad H_k(x):=f(x)\exp(\im kg(x)).
$$
\end{itemize}
\subsection{Error analysis}\label{sub:generr}
Accuracy of an MFCC rule is actually affected by the error of the involved interpolation method. In this subsection, we study how the interpolation error affects the total error of an MFCC rule. The results are then employed to choose efficient interpolation methods, such that the total error of the MFCC rule decays rapidly.

For any integer $m\ge0$ and a function $\varphi$ defined on $[a,b]$, consider the weighted seminorm
\begin{equation}
\|\varphi\|_{H_w^{m}[a,b]}:=\left\{\int_{a}^b\dfrac{|\varphi^{(m)}(x)|^2}{\sqrt{(b-x)(x-a)}}\,\df x\right\}^{1/2},
\end{equation}
as introduced in~\cite{dom13}. Clearly, the  Hilbert space $H_w^{m}[a,b]$, induced by $\|.\|_{H_w^{m}[a,b]}$, contains $C^m[a,b]$.
\begin{lemma}\label{lem:cc}
Let $\mathsf{C}_N$ and $\mathsf{E}_N$ be $(N+1)\times(N+1)$ matrices with the entries
\begin{equation}\label{eq:c}
\mathsf{C}_N(n+1,j+1)=(2/N)\cos(jn\pi/N),\quad n,j=0,\ldots,N,
\end{equation}
and
\begin{equation}\label{eq:e}
\mathsf{E}_N(n+1,j+1)=\left\{
                   \begin{array}{ll}
                     1/2 & \hbox{ if $n=j\in\{0,N\}$,} \\[1ex]
                     1 & \hbox{ if $n=j\in\{1,\ldots,N-1\}$,} \\[1ex]
                     0 & \hbox{ otherwise.}
                   \end{array}
                 \right.
\end{equation}
\end{lemma}
Then, $\mathsf{A}_N:=\sqrt{N/2}\,\mathsf{E}^{1/2}_N\mathsf{C}_N\mathsf{E}^{1/2}_N$ is a symmetric orthogonal matrix.
\begin{proof}
Clearly, $\mathsf{A}_N$ is symmetric since $\mathsf{C}_N$ and $\mathsf{E}_N$ are so. On the other hand, $\mathsf{C}_N\mathsf{E}_N$ is nothing but the matrix of the DCT of type I. Thus, $(\mathsf{C}_N\mathsf{E}_N)^2=(2/N)I_N$, and then,
\begin{equation}
\mathsf{A}^{\top}_N\mathsf{A}_N=\frac{N}{2}\mathsf{E}^{1/2}_N\mathsf{C}_N\mathsf{E}_N\mathsf{C}_N\mathsf{E}_N\mathsf{E}^{-1/2}_N=I_N.
\end{equation}
\end{proof}

For all $r\in[0, 2]$, define
$$
\rho(r)=
\left\{
  \begin{array}{ll}
    r, & 0\le r\le1, \\[1ex]
    (5r-3)/2, & 1\le r\le2.
  \end{array}
\right.
$$
The following theorem provides an error estimate for the MFCC rules~\eqref{eq:mfcc}:
\begin{theorem}\label{thm:generr}
Assume that $f$ and $g$ are so smooth that ${F\in H_w^m[g(a),g(b)]}$ for some integer $m\ge1$. Then, for any $r\in[0,2]$ with ${m>\rho(r)}$, the error of the (${N+1}$)-point MFCC rule with $N\ge m-1$ is estimated as follows:
\begin{align}\label{eq:generr}
\left|I^{[a,b]}_k(f,g)-I^{[a,b]}_{k,N}(f,g))\right|&\le
C k^{-r}h^{m+1-r}N^{-m+\rho(r)}\nonumber\\
&+C'k^{-1}\sqrt{N}\sup_{\tau\in[-1,1]}\left|(\widetilde{F}-\widetilde{Q}_{N'}\widetilde{F})(\tau)\right|,
\end{align}
where $h:=b-a$ is the length of the integration interval, and $C, C'$ are constants independent of $N$ and $N'$.
\end{theorem}
\begin{proof}
In the worst case when $I^{[a,b]}_k(f,g)$ is oscillatory, \ie $\tilde{k}\ge1/2$, the total error is bounded by
\begin{align}
\left|I^{[a,b]}_k(f,g)-l\exp(\im k c)I_{\tilde{k},N}(\widetilde{Q}_{N'}\widetilde{F})\right|&\le\left|I^{[a,b]}_k(f,g)-I^{[g(a),g(b)]}_{k,N}(F)\right| \nonumber\\[1ex]
&+l\left|I_{\tilde{k},N}(\widetilde{F}-\widetilde{Q}_{N'}\widetilde{F})\right|.\label{eq:toterr}
\end{align}
(If $\tilde{k}<1/2$, the second term in the right-hand-side of~\eqref{eq:toterr} drops.)

By assumption, $F\in H_w^m[g(a),g(b)]$. According to Theorem~2.6 of~\cite{dom13},
\begin{equation}\label{eq:thm26}
\left|I^{[a,b]}_k(f,g)|-I^{[g(a),g(b)]}_{k,N}(F)\right|\le C\sigma_{m,N} k^{-r}l^{m+1-r}N^{-m+\rho(r)},
\end{equation}
where $C$ is a constant independent of $m$ and $N$, and
\begin{equation}
\sigma_{m,N}=\left\{
               \begin{array}{ll}
                 \prod_{j=1}^{m-1}\left(\frac{N+1}{N+1-j}\right), & m>1, \\[1ex]
                 1, & m=1.
               \end{array}
             \right.
\end{equation}
For a fixed $m$, one can see that ${\sigma_{m,N}\to1}$ as ${N\to\infty}$. Hence, $C\sigma_{m,N}$ is bounded by a constant independent of $N$, and then by the mean value theorem on $g$,
\begin{align}
\left|I^{[a,b]}_k(f,g)|-I^{[g(a),g(b)]}_{k,N}(F)\right|&\le C k^{-r}l^{m+1-r}N^{-m+\rho(r)}\nonumber\\[2ex]
&\le C k^{-r}h^{m+1-r}N^{-m+\rho(r)}.\label{eq:errcor27}
\end{align}
where $C$ is independent of $N$.

On the other hand,
\begin{equation}\label{eq:wcdelta}
\left|I_{\tilde{k},N}(\widetilde{F}-\widetilde{Q}_{N'}\widetilde{F})\right|=|\boldchr{\omega}^{\top}_N\mathsf{C}_N\boldchr{\delta}_N|,
\end{equation}
where $\mathsf{C}_N$ is defined by~\eqref{eq:c}, and $\boldchr{\omega}_N$, $\boldchr{\delta}_N$ are column vectors defined by
\begin{align*}
\boldchr{\omega}_N&=[\omega_0(\tilde{k})/2,\omega_1(\tilde{k}),\ldots,\omega_{N-1}(\tilde{k}),\omega_N(\tilde{k})/2]^{\top},\\[1ex]
\boldchr{\delta}_N&=\left[\frac{(\widetilde{F}-\widetilde{Q}_{N'}\widetilde{F})(t_{0,N})}{2},(\widetilde{F}-\widetilde{Q}_{N'}\widetilde{F})(t_{1,N}),\ldots,
(\widetilde{F}-\widetilde{Q}_{N'}\widetilde{F})(t_{N-1,N}),\frac{(\widetilde{F}-\widetilde{Q}_{N'}\widetilde{F})(t_{N,N})}{2}\right]^{\top}.
\end{align*}
(see Remark 2.1 of~\cite{dom11}). By Cauchy-Schwarz inequality, Eq.~\eqref{eq:wcdelta} yields
\begin{equation}\label{eq:csi}
\left|I_{\tilde{k},N}(\widetilde{F}-\widetilde{Q}_{N'}\widetilde{F})\right|\le\|\boldchr{\omega}_N\|_2\|\mathsf{C}_N\boldchr{\delta}_N\|_2.
\end{equation}
The well-known asymptotic expansion of highly oscillatory integrals in the absence of stationary points (see, \eg\cite{ise05}) yields ${|I_{\tilde{k}}(T_n)|=\mathcal{O}(\tilde{k}^{-1})}$, and then
\begin{equation}\label{eq:wn}
\|\boldchr{\omega}_N\|_2\le C'l^{-1}k^{-1}\sqrt{N}.
\end{equation}
Note that $\|T_n\|_{\infty}\le1$ for all $n$, so $C'$ in~\eqref{eq:wn} can be assumed to be independent of $N$ and $N'$. On the other hand, Lemma~\eqref{lem:cc} implies that
\begin{align}\label{eq:cnen}
\left\|\mathsf{C}_N\mathsf{E}_N\right\|_2\le\left\|\mathsf{E}^{-1/2}_N\right\|_2\left\|\mathsf{E}^{1/2}_N\mathsf{C}_N\mathsf{E}^{1/2}_N\right\|_2\left\|\mathsf{E}^{1/2}_N\right\|_2
=\frac{2}{\sqrt{N}}.
\end{align}
Now, the upper bound~\eqref{eq:csi} can be estimated by~\eqref{eq:wn} and~\eqref{eq:cnen}:
\begin{align}
\left|I_{\tilde{k},N}(\widetilde{F}-\widetilde{Q}_{N'}\widetilde{F})\right|
&\le\|\boldchr{\omega}_N\|_2\left\|\mathsf{C}_N\boldchr{\delta}_N\right\|_2
\le\|\boldchr{\omega}_N\|_2\left\|\mathsf{C}_N\mathsf{E}_N\right\|_2\left\|\boldchr{y}_N\right\|_2\nonumber\\
&\le\frac{2}{\sqrt{N}}\|\boldchr{\omega}_N\|_2\left\|\boldchr{y}_N\right\|_2
\le C'l^{-1}k^{-1}\sqrt{N}\sup_{\tau\in[-1,1]}\left|(\widetilde{F}-\widetilde{Q}_{N'}\widetilde{F})(\tau)\right|, \label{eq:csi2}
\end{align}
where $\boldchr{y}_N:=\mathsf{E}^{-1}_N\boldchr{\delta}_N$. Now, the result follows by~\eqref{eq:toterr},~\eqref{eq:errcor27}, and~\eqref{eq:csi2}.
\end{proof}

The following result follows immediately from Theorem~\ref{thm:generr}:
\begin{corollary}\label{cor:generr}
Assume that $f$ and $g$ are so smooth that ${F\in H^{N+1}_w[g(a),g(b)]}$ for some integer $N\ge1$. Then, for any ${r\in[0,2]}$, the error of the (${N+1}$)-point MFCC rule is estimated as follows:
\begin{align}\label{eq:generr2}
\left|I^{[a,b]}_k(f,g)-I^{[a,b]}_{k,N}(f,g))\right|&\le
C k^{-r}h^{N+2-r}\frac{N^{-1+\rho(r)}}{N!2^N}\|g'\|_{\infty}^N\nonumber\\[1ex]
&+C'k^{-1}\sqrt{N}\sup_{\tau\in[-1,1]}\left|(\widetilde{F}-\widetilde{Q}_{N'}\widetilde{F})(\tau)\right|,
\end{align}
where $h:=b-a$, $C, C'$ are constants independent of $N, N'$, and ${\|g'\|_{\infty}=\sup_{x\in[a,b]}\left|g'(x)\right|}$.
\end{corollary}
\begin{proof}
Note that
\begin{equation}
\sigma_{N+1,N}=\frac{N^N}{N!}\left(1+\frac{1}{N}\right)^N,
\end{equation}
and ${\left(1+1/N\right)^N}$ is bounded for all ${N\ge1}$. Thus, by the mean value theorem on $g$,~\eqref{eq:thm26} with ${m=N+1}$ is adjusted as
\begin{align*}\label{eq:thm26}
\left|I^{[a,b]}_k(f,g)|-I^{[g(a),g(b)]}_{k,N}(F)\right|&\le C k^{-r}\left(\frac{h\|g'\|_{\infty}}{2}\right)^{N+2-r}\frac{N^{-1+\rho(r)}}{N!}\\
&=C k^{-r}h^{N+2-r}\frac{N^{-1+\rho(r)}}{N!2^N}\|g'\|_{\infty}^N.
\end{align*}
\end{proof}

As it is seen by~\eqref{eq:generr} or~\eqref{eq:generr2}, the total error of the MFCC rule is bounded by the sum of two terms: The first term is nothing but the error of the FCC rule, studied in details in~\cite{dom11,dom13}; it does not depend on $N'$ and decays rapidly as $k$ or $N$ grows. The second term associates with the interpolation process, that may be quite large due to the Runge effect. In the sequel, we present two strategies in order to treat the latter problem and obtain error bounds which decay rapidly by increasing $N$ or $N'$.
\section{No stationary points}\label{sec:nonstat}
One popular idea for treating the Runge phenomenon lies in splines. In this section, we present a method of integration based on \emph{composite} MFCC rules on uniform meshes. Without loss of generality, we consider the integral~\eqref{eq:intfg} over the interval $[0,1]$. Recall from the previous section the assumption that ${g'(x)>0}$, ${x\in[0,1]}$.
\subsection{Algorithm~I}\label{sub:algI}
Take a large integer $M$ and divide the integration interval $[0,1]$ into $M$ subintervals of equal lengths $h$:
\begin{equation}\label{eq:ugrid}
x_n:=nh,\quad n=0,\ldots,M,\quad h:=1/M.
\end{equation}
Then,
$$
g(0)=g(x_0)<\cdots<g(x_M)=g(1).
$$
Assume that $f\in C^{N+1}[0,1]$ and $g\in C^{N+2}[0,1]$ for some integer $N\ge 1$. Then it is easy to see that $F$, as defined by~\eqref{eq:F}, lies in $C^{N+1}[g(0),g(1)]$. In each panel $[x_{n-1},x_n]$, we employ the (${N+1}$)-point MFCC rule~\eqref{eq:mfcc} with $N'=N$, $u_j=t_{N-j,N}$, and $\widetilde{Q}_N$ being the Lagrange interpolating projection at the points $-1=d_0<\cdots<d_N=1$. The accuracy of this kind of composite MFCC rules may be studied as follows.
\subsubsection{Error analysis}
Take an arbitrary subinterval $[x_{n-1},x_n]$, and for simplification of notations, set $[a,b]:=[x_{n-1},x_n]$. We divide our discussion into two parts: At first, we show that $d_0,\ldots,d_N$ are `good' interpolation points for small $[a,b]$ in the sense that the corresponding Lebesgue constant grows only logarithmically by $N$ provided that $M$ increases according to $N$. In the second part, the total error of the composite MFCC rule is estimated by Corollary~\ref{cor:generr} and an error bound for the interpolant $\widetilde{Q}_N\widetilde{F}$ associated with each panel.
\paragraph{Estimating the Lebesgue constant}
By the mean value theorem,~\eqref{eq:d} yields
\begin{equation}\label{eq:d-d}
d_i-d_j=\dfrac{b-a}{g(b)-g(a)}g'(\eta_{ij})(t_{N-i,N}-t_{N-j,N}),
\end{equation}
for any $i,j\in\{0,\ldots,N\}$ and some $\eta_{ij}\in(a,b)$. Since ${d_0=t_{N,N}=-1}$ and ${d_N=t_{0,N}=1}$, \eqref{eq:d-d} implies that the points $d_0,\ldots,d_N$ can be considered as `perturbed Clenshaw-Curtis points' when $M$ is sufficiently large. This is because $g'$ is positive and uniformly continuous on $[a,b]$, so
\begin{equation}\label{eq:mvt}
\dfrac{b-a}{g(b)-g(a)}g'(\eta_{ij})\simeq1,
\end{equation}
by the mean value theorem.

The Lebesgue constant for a set $\Pi_N$ of ${N+1}$ interpolation points in an interval $\mathcal{I}$ is defined by
\begin{equation}
\max_{x\in\mathcal{I}}\sum_{i=0}^N|L_i(x)|,
\end{equation}
where $L_i$ is the ($i+1$)-th Lagrange cardinal polynomial for the set $\Pi_N$. It is well-known that the Lebesgue constant determines how relatively close is the maximum error of an interpolation to that of the best approximation. Denote by $\Lambda_N$ and $\tilde{\Lambda}_N$ the Lebesgue constants associated with the Clenshaw-Curtis points ${t_{N,N},\ldots,t_{0,N}}$ and the perturbed Clenshaw-Curtis points ${d_0,\ldots,d_N}$, respectively. We show that $\tilde{\Lambda}_N\le\mathrm{e}\Lambda_N$ provided that $M$ increases according to $N$.

Considering the change of variables $\tau=\ell^{-1}_2(g(\ell_1(x)))$,~\eqref{eq:d-d} implies that the ($i+1$)-th Lagrange cardinal polynomial $\tilde{L}_i$ associated with the points $d_0,\ldots,d_N$ can be written as follows:
\begin{equation}\label{eq:lag}
\tilde{L}_i(\tau)=\prod_{\btop{j=0}{j\neq i}}^N\frac{\tau-d_j}{d_i-d_j}=\prod_{\btop{j=0}{j\neq i}}^N\frac{g'(\eta_j(\tau))}{g'(\eta_{ij})}\frac{x-t_{N-j,N}}{t_{N-i,N}-t_{N-j,N}},
\end{equation}
where $\eta_j(\tau)\in(a,b)$ for all $j=0,\ldots,N$ and $\tau\in[-1,1]$. Positivity and continuity of $g'$ on the compact interval $[0,1]$ imply that there exists $h_N>0$ such that for all subintervals $[a,b]$ with $b-a<h_N$,
\begin{equation}\label{eq:ggeps}
1-\frac{1}{N}<\frac{g'(\eta_j(\tau))}{g'(\eta_{ij})}<1+\frac{1}{N},\quad \tau\in[-1,1],\quad i,j=0,\ldots,N.
\end{equation}
Thus,~\eqref{eq:lag} with~\eqref{eq:ggeps} implies that,
\begin{equation}\label{eq:lagineq}
\left(1-1/N\right)^N\sum_{i=0}^N|L_i(x)|<\sum_{i=0}^N|\tilde{L}_i(\tau)|<\left(1+1/N\right)^N\sum_{i=0}^N|L_i(x)|,
\end{equation}
where $L_i$ is the ($i+1$)-th Lagrange cardinal polynomial for the Clenshaw-Curtis points. Since $\tau=\ell^{-1}_2(g(\ell_1(x)))$ is a bijective mapping on $[-1,1]$, taking maximum over $x\in[-1,1]$ (or over $\tau\in[-1,1]$),~\eqref{eq:lagineq} yields
\begin{equation}\label{eq:ineqleb}
\left(1-1/N\right)^N\Lambda_N<\tilde{\Lambda}_N<\left(1+1/N\right)^N\Lambda_N.
\end{equation}
Since ${(1+1/N)^N}$ is bounded,~\eqref{eq:ineqleb} implies that the Runge effect never occurs in the interpolation of $\widetilde{F}$ at the points ${d_0,\ldots,d_N}$, and the error is comparable to that of Chebyshev interpolation provided that $M\ge1/h_N$, i.e., $M$ is taken large enough according to $N$.
\paragraph{Total error bound}
Clearly, the interpolating error is bounded by
$$
\sup_{x\in[-1,1]}\left|(\widetilde{F}-\widetilde{Q}_N\widetilde{F})(x)\right|\le (2l)^{N+1}\max_{x\in[g(a),g(b)]}\left|F^{(N+1)}(x)\right|.
$$
This is because $F\in C^{N+1}[g(a),g(b)]$, and then $\widetilde{F}\in C^{N+1}[-1,1]$. Now, by the mean value theorem on $g$,
\begin{equation}
\sup_{x\in[-1,1]}\left|(\widetilde{F}-\widetilde{Q}_N\widetilde{F})(x)\right|\le C'h^{N+1},
\end{equation}
for $N$ large enough. Here, $C'$ is a generic constant independent of $M$, but dependent of $N$. Since $\widetilde{F}\in C^{N+1}[-1,1]\subseteq H^{N+1}_w[-1,1]$, Corollary~\ref{cor:generr} can now be employed to estimate the total error in the arbitrary panel $[a,b]$:
\begin{align}
\left|I^{[a,b]}_k(f,g)|-l\exp(\im k c)I_{\tilde{k},N}(\widetilde{Q}_N\widetilde{F})\right|&\le C k^{-r}h^{N+2-r}\frac{N^{-1+\rho(r)}}{N!2^N}\|g'\|_{\infty}^N\nonumber\\
&{}+C'k^{-1}h^{N+1},\label{eq:err1raw}
\end{align}
where $C$ is independent of $M$ and $N$. If one takes sum over all the subintervals ${[x_{n-1},x_n]}$ and take into account \sloppy that ${F\in C^{N+1}[g(0),g(1)]}$, an error bound for Algorithm~I is obtained by~\eqref{eq:err1raw}:
\begin{equation}\label{eq:err1}
C k^{-r}h^{N+1-r}\frac{N^{-1+\rho(r)}}{N!2^N}\|g'\|_{\infty}^N+C'k^{-1}h^N.
\end{equation}
\begin{remark}
Note that $C'$ in the error bound~\eqref{eq:err1} may grows by $N$. Thus, the recommended strategy in Algorithm~I is to keep $N$ fixed at a moderate value and increase $M$ until the required accuracy is achieved. This can in addition suppress the Runge effect as discussed above.
\end{remark}

A few words about the complexity of the algorithm is also necessary. The (${N+1}$)-point FCC rule can be stably implemented with $\mathcal{O}(N\ln N)$ flops~\cite{dom11}. The Lagrange interpolation can also be performed stably by an improved formula, the so-called the ``first form of the barycentric interpolation formula'' (see~\cite{rut90,ber04,hig04}); evaluation of the formula at any point costs at $\mathcal{O}(N^2)$, so evaluation of all the values ${\widetilde{Q}\widetilde{F}(t_{j,N})}$, ${j=0,\ldots,N}$, requires $\mathcal{O}(N^3)$ flops. Since this cost should be paid for each panel, and there exist $M$ panels indeed, Algorithm~I requires ${\mathcal{O}\left(MN(N^2+\ln N)\right)}$ flops.
\subsubsection{Numerical experiments}\label{sub:numericI}
Throughout this subsection, we consider the model integral $I^{[0,1]}_k(f,g)$ with
\begin{equation}\label{eq:sampfg1}
f(x)=\dfrac{x^{4.5}}{1+x^2},\qquad g(x)=\sqrt{x^2+3x+4},
\end{equation}
with some $k>0$. It can be seen that $g'(x)>0$, $x\in[0,1]$, $f\in C^{4}[0,1]$ and $g\in C^{\infty}[0,1]$. We employ Algorithm~I and illustrate the error estimate~\eqref{eq:err1}.

\begin{table}
\renewcommand{\arraystretch}{1.4}
\caption{Application of Algorithm~I to the model integral ${I^{[0,1]}_{100}(f,g)}$ with $f$ and $g$ defined by~\eqref{eq:sampfg1}: Absolute error with its decaying rate.\label{tbl:one}}
\begin{center} \footnotesize
\begin{tabular}{c*{6}{c}}
\cline{2-7}
& \multicolumn{2}{c}{$N=1$} & \multicolumn{2}{c}{$N=2$} & \multicolumn{2}{c}{$N=3$} \\
\cline{1-7}
$M$& error & rate & error & rate & error & rate \\
\cline{1-7}
2  & $1.87\times10^{-4}$ &    --- & $1.27\times10^{-5}$ &    --- & $2.22\times10^{-6}$ &    ---  \\
4  & $2.72\times10^{-5}$ &    2.8 & $3.87\times10^{-6}$ &    1.7 & $4.43\times10^{-7}$ &    2.3  \\
8  & $3.42\times10^{-5}$ &   -0.3 & $8.11\times10^{-8}$ &    5.6 & $3.50\times10^{-8}$ &    3.7  \\
16 & $3.98\times10^{-5}$ &   -0.2 & $1.40\times10^{-7}$ &   -0.8 & $1.41\times10^{-9}$ &    4.6  \\
32 & $3.69\times10^{-6}$ &    3.4 & $2.37\times10^{-9}$ &    5.9 & $1.66\times10^{-11}$ &    6.4  \\
64 & $8.25\times10^{-7}$ &    2.2 & $1.25\times10^{-10}$ &   4.2 & $7.41\times10^{-13}$ &    4.5  \\
\cline{1-7}
\end{tabular}
\end{center}
\end{table}
\paragraph{Experiment~1}
Here we study the case when $k$ and $N$ are kept fixed and $h\to0$. According to the error bound~\eqref{eq:err1}, the convergence of the method is of order  $\mathcal{O}(h^N)$, which can be illustrated by the following example. Let $k=100$. Since $f\in C^{4}[0,1]$ and $g\in C^{\infty}[0,1]$, $F\in C^{4}[g(0),g(1)]$.  Thus, the error bound~\eqref{eq:err1} is valid if the parameter $N$ does not exceed 3. The absolute error with its decaying rate, as $h\to0$, is given in Table~\ref{tbl:one} for $N=1,2,3$. As mentioned in~\S\ref{sub:algI}, the distribution of the interpolation nodes in each panel is so that the corresponding Lebesgue constant is rather small. Thus, the convergence rate will be higher than what we expect. For coarser grids (corresponding to larger $h$), instead, the convergence rate of ${\mathcal{O}(h^N)}$ is not expected. Considering all these facts together, the numerical results of Table~\ref{tbl:one} are in agreement with the theory.

\begin{figure}
\includegraphics[width=\textwidth]{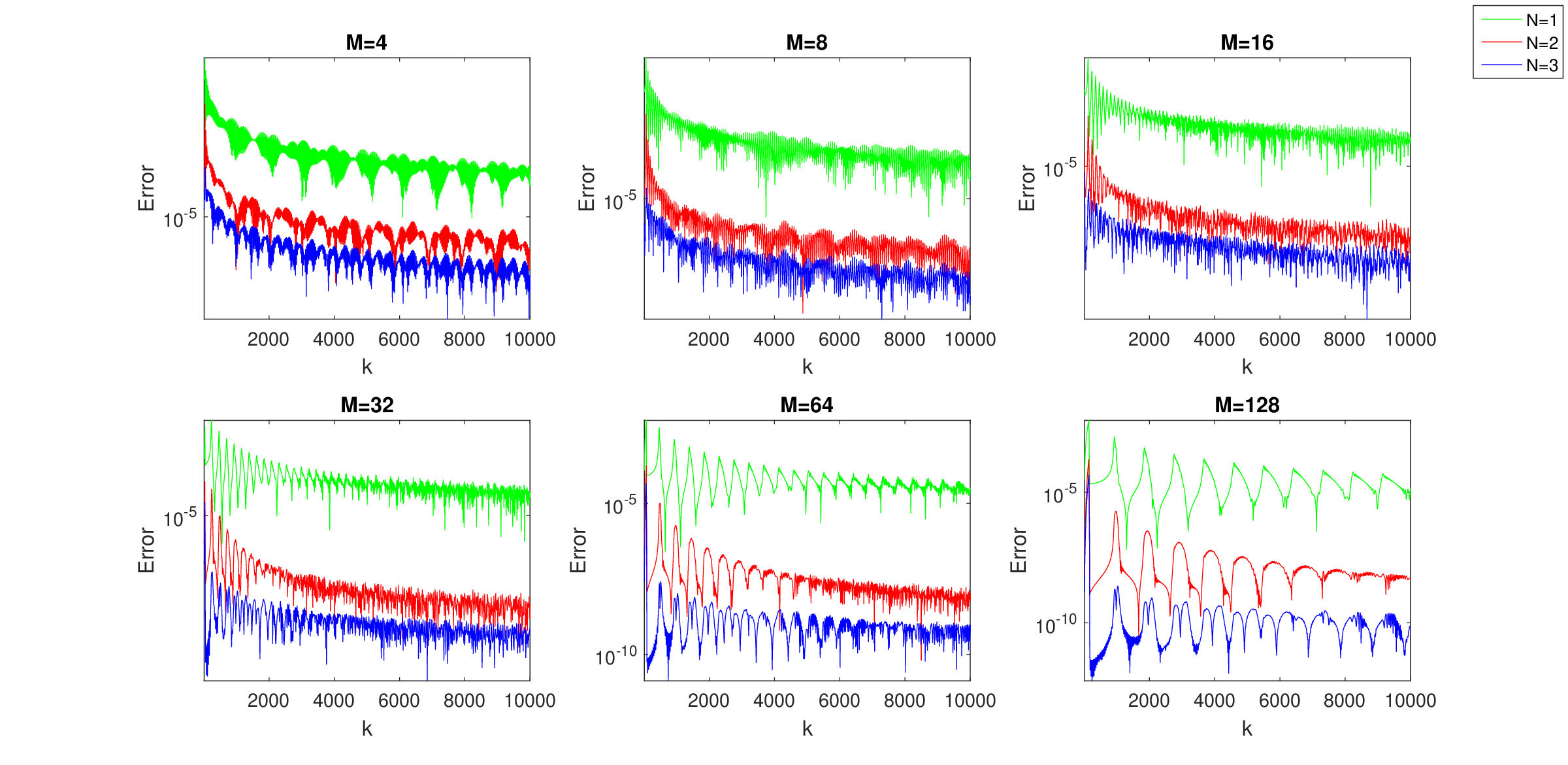}
\caption{\small Absolute error scaled by $k$ of Algorithm~I with some $M$ and $N$, when applied to the model integral $I^{[0,1]}_k(f,g)$ with~\eqref{eq:sampfg1}. \label{fig:meth1}}
\end{figure}
\paragraph{Experiment~2}
The error bound~\eqref{eq:err1} suggests that for a fixed parameters $N$ and $M$, the maximum rate of convergence is ${\mathcal{O}(k^{-1})}$ as ${k\to\infty}$. In this experiment, we illustrate this rate by a numerical example. For each $M\in\{4,8,16,32,64,128\}$ and $N\in\{1,2,3\}$, we apply Algorithm~I to the model integral ${I^{[0,1]}_k(f,g)}$ with~\eqref{eq:sampfg1} when $k$ varies in a wide band from 10 to $10^4$. The absolute error is scaled by $k$, and the scaled error for each $N$ and $M$ has been plotted as a function of $k$ (Figure~\ref{fig:meth1}). As it is seen the error for each $M$ and $N$ does not deteriorate as $k$ increases, and this observation is in agreement with the maximum convergence order ${\mathcal{O}(k^{-1})}$ suggested by the theory.

\begin{figure}
\includegraphics[width=\textwidth]{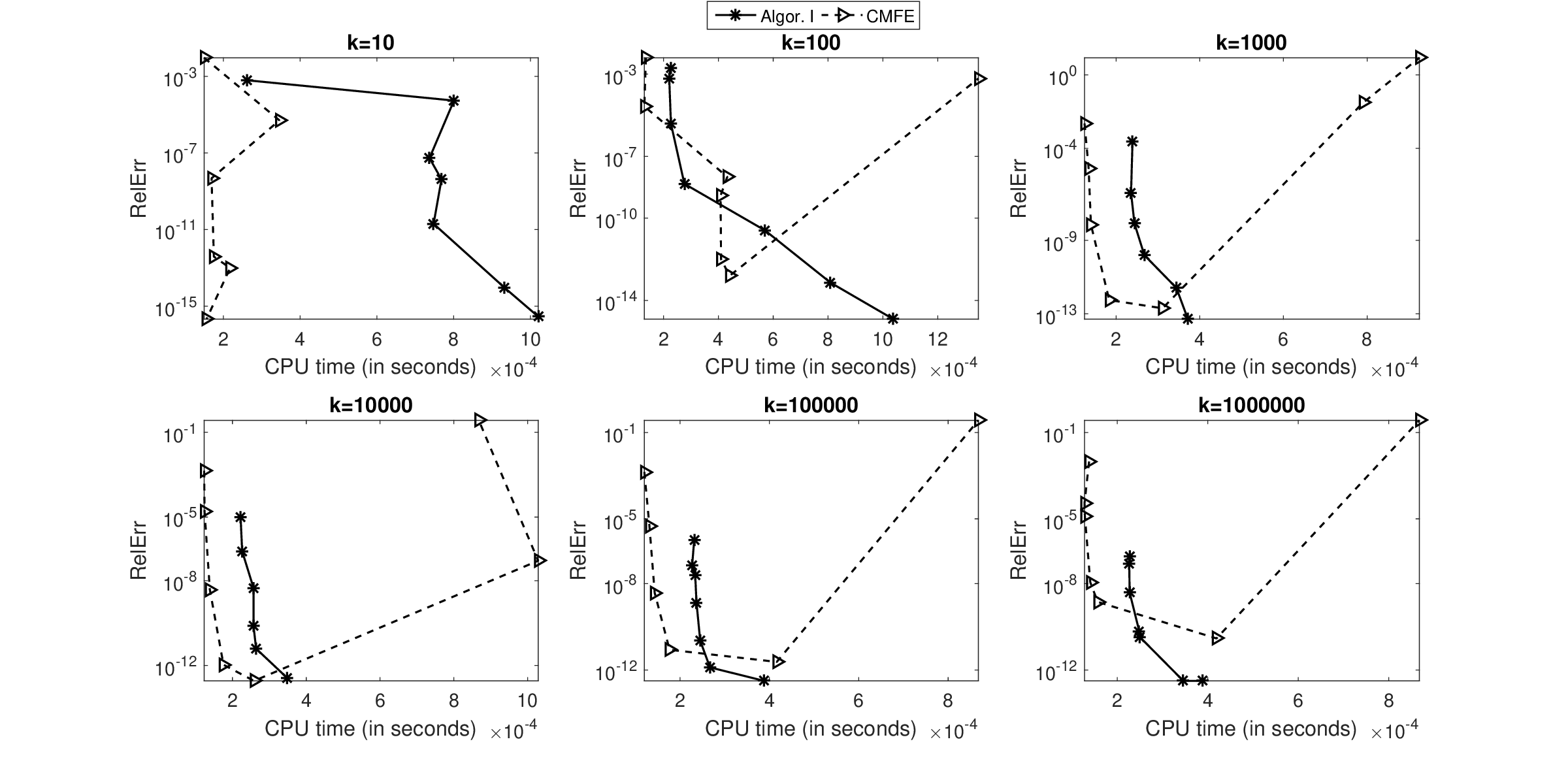}
\caption{\small Algorithm~I (solid line) and the CMFE rule of~\cite{ma18} with exponential order of convergence (dashed line) when applied to the integral~\eqref{eq:sampleMaXu}: Relative error vs execution time (in seconds).\label{fig:maxu}}
\end{figure}
\paragraph{Experiment~3}
Here, we compare convergence speed of Algorithm~I with the composite Filon-type rules~\cite{ma18} for the integrals ${I^{[0,1]}_{k}(f,g)}$ with smooth $f$ and $g$. For a given integer ${n\ge1}$, consider the graded mesh ${x_0=0}$, ${x_j=k^{(j-1)/(n-1)-1}}$, ${j=1,\ldots,n}$. For the $j$-th panel ${[x_{j-1},x_j]}$, define the parameters ${M_j:=\left\lceil\max\{|g'(x_{j-1})|,|g'(x_j)|\}\right\rceil}$ and ${N_j:=n(n-1)/(n+1-j)}$, where $\lceil x\rceil$ denotes the smallest integer not less than $x$. Then, apply Algorithm~I to the integral over ${[x_{j-1},x_j]}$ with ${N_j+1}$ Clenshaw-Curtis points and a uniform mesh with $M_j$ panels. This is a special kind (with Clenshaw-Curtis points) of the rule CMFE proposed in~\cite{ma18} for ${I^{[0,1]}_{k}(f,g)}$ with smooth $f$. In the absence of harmful rounding, the rule converges exponentially. More precisely, when ${n\to\infty}$, the error of the rule scales as ${\mathcal{O}\left((n-1)^{-1/2}k^{-n-1}\right)}$ provided that
\begin{equation}
(n-1)\left(\ln(n-2+\mathrm{e})-1\right)\ge\ln k.
\end{equation}

We apply Algorithm~I and the rule CMFE, as described above, to the following sample integral for $k=10^d$, $d=1,\ldots,6$:
\begin{equation}\label{eq:sampleMaXu}
\int_0^1\exp\left(\im k(\sin(\pi x/2)+2x)/3\right)\,\df x.
\end{equation}
Figure~\ref{fig:maxu} shows the relative error vs the execution time (in seconds) for each case. In this example, the CMFE rule is in general faster than Algorithm~I because it requires smaller number of subintervals to reach a given accuracy. This is due to the special graded meshes which result in the exponential order of convergence for the CMFE rule. When higher accuracy demanded, however, it may exhibit instable behavior due to rounding error in the first few panels. Indeed, the panels in a narrow neighborhood of 0 become too small for larger $k$. For example, one can see that the first panel is ${[0,1/k]}$, which in turn partitioned by a uniform mesh of size $M_1$ and in each subpanel, ${N_1:=n-1}$ interpolation points are involved. Thus, one should evaluate functions at a bunch of very closed points, and this may deteriorate the accuracy due to rounding. In the sample integral~\eqref{eq:sampleMaXu}, $f$ is constant; if it was not so, this effect might occur even more apparently.
\subsection{Algorithm~II}\label{sub:algII}
In this section, we follow~\cite{bru12} and employ an efficient interpolation method which does not allow the Runge effect. In addition, one can always increase its accuracy to as high as desired without increasing the computational cost.

Consider the integral~\eqref{eq:intfg} on an arbitrary interval $[a,b]$ (not necessarily of small length). Here, we do not divide the integration interval $[a,b]$ into smaller subintervals as in Algorithm~I. Instead, we consider the (${N+1}$)-point MFCC rule on the whole interval and approximate the values ${\widetilde{F}(t_{j,N})}$ by interpolating $\widetilde{F}$ at a limited number of points, selected from $\{d_0,\ldots,d_{N'}\}$, in such a way that they surround $t_{j,N}$.

For a large integer $N'$, choose a set of arbitrary points ${-1=u_0<\cdots<u_{N'}=1}$ and consider the points ${-1=d_0<\ldots<d_{N'}=1}$ as defined by~\eqref{eq:d}. Choose a positive integer ${s<N'}$ and fix it throughout the process. For each ${j=0,\ldots,N}$, select an $s$-tube ${\mathcal{N}_{n,s}=(d_n,\ldots,d_{n+s-1})}$ such that ${d_n\le t_{j,N}\le d_{n+s-1}}$. Note that $n$ depends (not uniquely) on $j$. Thus, such a selection is not necessarily unique, while it is always possible since ${s<N'}$. In~\cite{bru12}, some selecting strategies are introduced, and in~\cite{maj15} a \matlab\ code of a certain selecting strategy has been provided. Now, consider the approximation
\begin{equation}
\widetilde{F}(t_{j,N})\approx p_{n,s}(t_{j,N}),\quad j=0,\ldots,N,
\end{equation}
where $p_{n,s}$ is the Lagrange interpolation polynomial of $\widetilde{F}$ at the points of $\mathcal{N}_{n,s}$. This means that $p_{n,s}$ is of the degree ${s-1}$, and it may change by $j$.

For computing $p_{n,s}(t_{j,N})$, we recommend the so-called ``first form of the barycentric interpolation formula''~\cite{rut90} that is a reformulation of the Lagrange interpolation polynomials. In~\cite{hig04} it has been proved that the formula is backward stable for any set of interpolation points.

Based on~\eqref{eq:mfcc}, one can introduce the rule
\begin{equation}\label{eq:rule2}
I^{[a,b]}_k(f,g)\approx l\exp(\im kc)\sum_{n=0}^N{}''\alpha_{n,N}(p_{n,s})\omega_n(\tilde{k}),
\end{equation}
where ${(\alpha_{n,N}(p_{n,s}))_{n=0}^N}$ is the DCT of type I of \sloppy ${(p_{n,s}(t_{j,N}))_{j=0}^N}$.
\subsubsection{Error analysis}
If $f\in C^s[a,b]$ and $g\in C^{s+1}[a,b]$, $\widetilde{F}\in C^s[-1,1]$, and then
\begin{equation}\label{eq:errns}
\max_{0\le n\le N}|(\widetilde{F}-p_{n,s})(t_{j,N})|\le C'\lambda^s,
\end{equation}
where
\begin{equation}
\lambda=\max_{1\le j\le N'}|d_j-d_{j-1}|,
\end{equation}
and $C'$ is independent of $N$ and $N'$, but dependent of $s$ (see, e.g.,~\cite{bru12}).

On the other hand, ${F\in C^s[g(a),g(b)]\subseteq H_w^s[g(a),g(b)]}$. Thus, for any ${r\in[0,2]}$, ${s>\rho(r)}$, and $N\ge s-1$, one can obtain the total error of the rule~\eqref{eq:rule2} by Theorem~\ref{thm:generr} and~\eqref{eq:errns}:
\begin{equation}\label{eq:err2}
C k^{-r}(b-a)^{s+1-r}N^{-s+\rho(r)}+C'k^{-1}\lambda^s\sqrt{N},
\end{equation}
where $C$ is independent of $N$, $N'$.

One can easily see from~\eqref{eq:d-d} that $\lambda$ decreases with the rate of ${\mathcal{O}(1/N')}$ as ${N'\to\infty}$. Therefore, if we take ${N'\ge kN}$, the error bound~\eqref{eq:err2} is reduced to
\begin{equation}\label{eq:err2p}
C k^{-r}(b-a)^{s+1-r}N^{-s+\rho(r)}+C'k^{-s-1}N^{-s+1/2}.
\end{equation}

The discussion about the complexity of the method is the same as that of Algorithm~I. Since the total interpolation process requites ${\mathcal{O}(Ns^2)}$ flops, the whole algorithm is performed at the cost of ${\mathcal{O}\left(N(\ln N+s^2)\right)}$. In order to increase the accuracy of the interpolation, one needs to increase $N'$ only, while the parameter $s$ is usually kept fixed at a moderate value. In practice, one should compute only $Ns$ members of $\{d_0,\ldots,d_{N'}\}$, not all of them. Thus, the cost of computation does not grow by $N'$. Also, the (${N+1}$)-point FCC rule converges rapidly as $N$ increases, so by a moderate value of $N$, one can reach a rather high accuracy. Therefore, the integral~\eqref{eq:intfg} can be accurately approximated by Algorithm~II at a rather low cost. We carry out a set of numerical experiments in Section~\ref{sec:comp} to illustrate our claim here.
\subsubsection{Numerical experiments}\label{sub:numericII}
Throughout this subsection, we consider the model integral $I^{[-1,1]}_k(f,g)$ with
\begin{equation}\label{eq:sampfg2}
f(x)=\frac{x-1}{1+x^2},\qquad g(x)=\sqrt{x^2+3x+4},
\end{equation}
and some $k$. It is easy to see that $g'(x)>0$, $x\in[-1,1]$ and $f,g\in C^s[-1,1]$ for all $s\ge1$. Thus, Algorithm~II can be applied. The aim of this subsection is to illustrate the convergence estimate~\eqref{eq:err2p}.

\begin{table}
\renewcommand{\arraystretch}{1.4}
\caption{Application of Algorithm~II with ${N'=100N}$ to the model integral ${I^{[-1,1]}_{100}(f,g)}$ with $f$ and $g$ defined by~\eqref{eq:sampfg2}: Absolute error with its decaying rate.\label{tbl:two}}
\begin{center} \footnotesize
\begin{tabular}{c*{8}{c}}
\cline{2-9}
& \multicolumn{2}{c}{$s=2$} & \multicolumn{2}{c}{$s=3$} & \multicolumn{2}{c}{$s=4$} & \multicolumn{2}{c}{$s=5$} \\
\cline{1-9}
$N$ & error & rate & error & rate & error & rate & error & rate \\
\cline{1-9}
2  & $6.42\times10^{-4}$ &    ---  & $6.42\times10^{-4}$ &    --- & $6.42\times10^{-4}$ &    --- & $6.42\times10^{-4}$ &    --- \\
4  & $5.42\times10^{-4}$ &    0.2  & $5.42\times10^{-4}$ &    0.2 & $5.42\times10^{-4}$ &    0.2 & $5.42\times10^{-4}$ &    0.2 \\
8  & $9.91\times10^{-5}$ &    2.5  & $9.91\times10^{-5}$ &    2.5 & $9.91\times10^{-5}$ &    2.5 & $9.91\times10^{-5}$ &    2.5 \\
16 & $2.93\times10^{-6}$ &    5.1  & $2.93\times10^{-6}$ &    5.1 & $2.93\times10^{-6}$ &    5.1 & $2.93\times10^{-6}$ &    5.1 \\
32 & $1.48\times10^{-8}$ &    7.6  & $1.77\times10^{-9}$ &   10.7 & $1.73\times10^{-9}$ &   10.7 & $1.73\times10^{-9}$ &   10.7 \\
64 & $1.21\times10^{-9}$ &    3.6  & $1.71\times10^{-12}$ &  10.0 & $2.22\times10^{-15}$ &  19.6 & $1.17\times10^{-15}$ &   20.5 \\
\cline{1-9}
\end{tabular}
\end{center}
\end{table}
\paragraph{Experiment~1}
If $k$ is kept fixed, the error bound~\eqref{eq:err2p} decays with the maximum rate of ${\mathcal{O}(N^{-s+1/2})}$ as $N$ increases. We consider the assumed model integral with ${k=100}$, and employ the algorithm with ${s=2,3,4,5}$, some increasing $N$, and ${N'=kN}$. The absolute error with its decaying rate is given in Table~\eqref{tbl:two}. As it is seen, the experimental convergence rate for each $s$ is in agreement with the theoretical one.

\begin{figure}
\includegraphics[width=\textwidth]{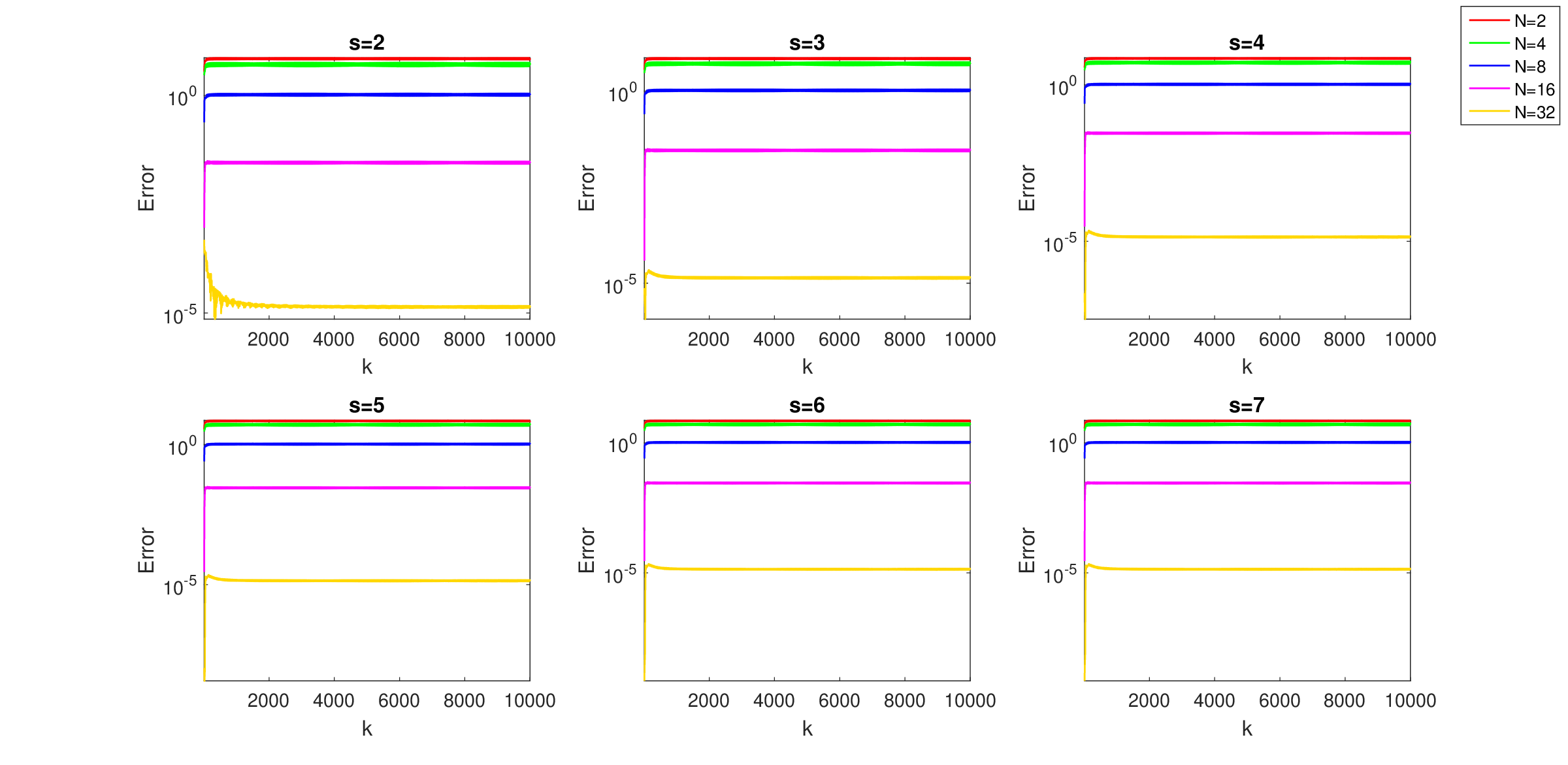}
\caption{\small Absolute error scaled by $k^2$ of Algorithm~II with some $s$ and $N$, when applied to the model integral ${I^{[-1,1]}_k(f,g)}$ with~\eqref{eq:sampfg2}. \label{fig:meth2}}
\end{figure}
\paragraph{Experiment~2}
If $N$, $N'$, and $s$ are kept fixed, the error estimate~\eqref{eq:err2p} suggests that the rate of convergence as $k\to\infty$ is not higher than ${\mathcal{O}(k^{-2})}$ provided that ${N'\ge kN}$. Here, we consider the model integral ${I^{[-1,1]}_k(f,g)}$ with~\eqref{eq:sampfg2} when $k$ varies in a wide band from 10 to $10^4$. Algorithm~II with ${s=2,\ldots,7}$ and ${N=2,4,8,16,32}$ is applied while ${N'=kN}$. The absolute error is scaled by $k^2$ and the scaled error for each $s$ and $N$ is plotted as a function of $k$ (see Figure~\ref{fig:meth2}). The horizontal trend of the scaled error, corresponding to each pair ${(s,N)}$, illustrates the maximum convergence rate of ${\mathcal{O}(k^{-2})}$.
\section{Comparisons}\label{sec:comp}
In this section, we compare Algorithms~I and~II to see which one in practice achieves a given accuracy faster. Our answer to this question, is based on some numerical results accompanied by a rough theoretical discussion.

Recall from Section~\ref{sub:algI} that Algorithm~I requires ${\mathcal{O}\left(MN(N^2+\ln N)\right)}$ flops. From the error bound~\eqref{eq:err1}, one can see that a rather large $N$ may reduce the error effectively even if ${h<1}$ is not too small. This is because $N$ is appearing in the exponents. Also, it is not recommended to use very large $N$ since the complexity of the algorithm increases rapidly. Thus, $N$ is better to take moderate values while $M$ should be rather small.

Similarly, we recall from Section~\ref{sub:algII} that the number of required flops for Algorithm~II is ${\mathcal{O}\left(N(\ln N+s^2)\right)}$. Note that the exponents in the error bound~\eqref{eq:err2p} may not be as large as those in the error bound~\eqref{eq:err1} because in practice, $s$ usually takes rather small values; larger values increases the complexity of the algorithm rapidly.

Based on the discussion above, our guess is that Algorithm~I may be faster than Algorithm~II provided that $M$ is rather small. In the following, we carry out a set of numerical experiments to support this guess.

\begin{figure}
\includegraphics[width=\textwidth]{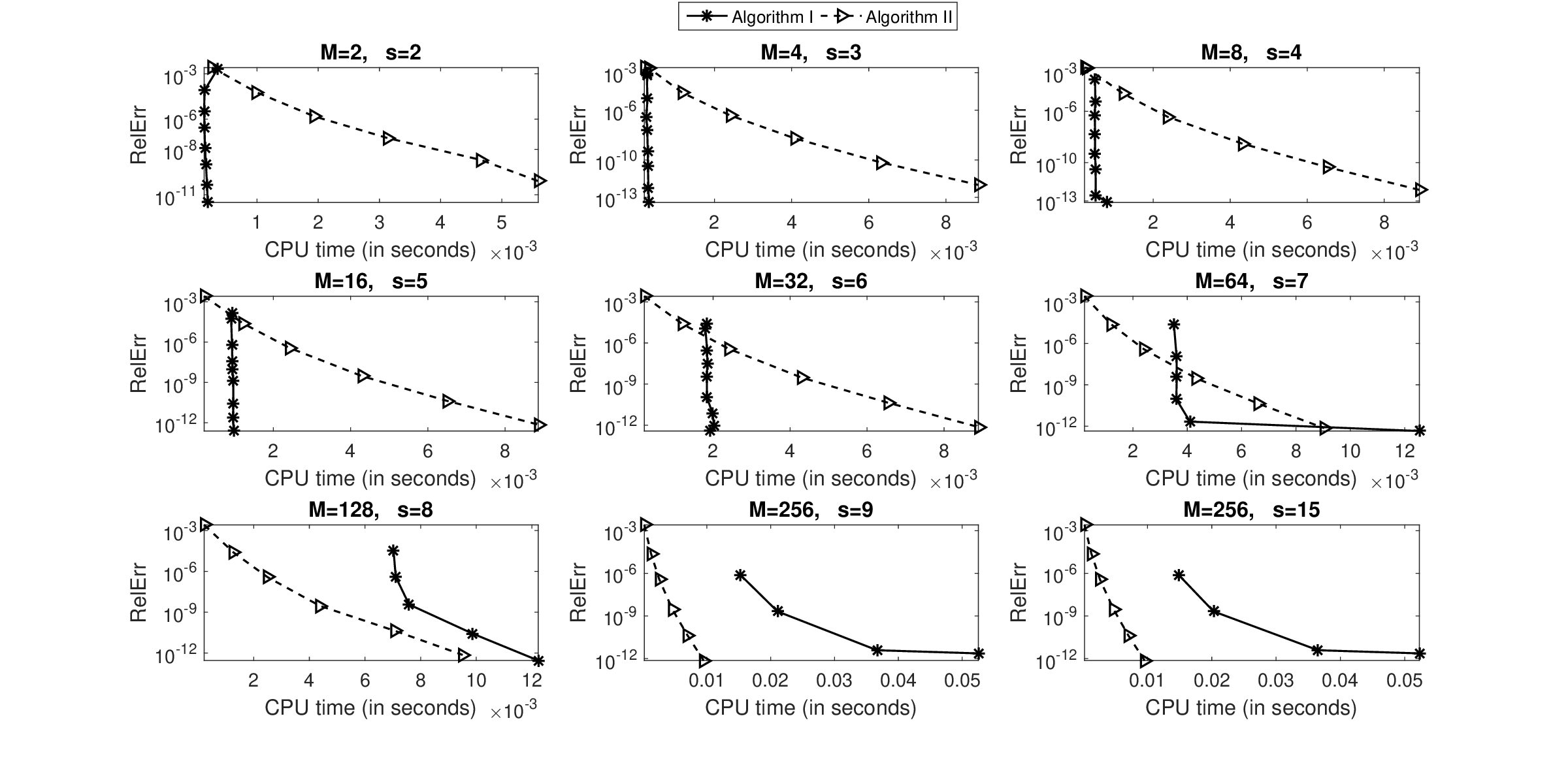}
\caption{\small Average CPU time (in seconds) vs relative error of Algorithms~I and~II when applied to the model integral~\eqref{eq:exmplcomp1}. The number $M$ of subintervals in Algorithm~I and the interpolation parameter $s$ in Algorithm~II are kept fixed in each plot while the accuracy increases by growing $N$. In Algorithm~II, ${N'=kN}$.\label{fig:comp1}}
\end{figure}
\paragraph{Experiment~1}
Recalling the model integral in~\S\ref{sub:numericII} with $k=1000$, consider the integral
\begin{equation}\label{eq:exmplcomp1}
\int_{-1}^1\frac{x-1}{1+x^2}\exp\left(1000\,\im\sqrt{x^2+3x+4}\right)\,\df x.
\end{equation}
In Algorithm~I, $M$ is kept fixed, and the accuracy increases by letting ${N\to\infty}$ (Figure~\ref{fig:comp1}). The algorithm for each set of the parameters ${\{M,N\}}$ is executed 10 times and the average of CPU times (in seconds) vs the relative error is plotted by a solid line. In Algorithm~II, $s$ is kept fixed and the accuracy increases by letting ${N\to\infty}$ while ${N'=kN}$. The algorithm for each set of the parameters ${\{s,N\}}$ is executed 10 times and the average of CPU times (in seconds) vs the relative error is plotted by a dashed line. Figure~\ref{fig:comp1} consists of nine subplots for some increasing values of $M$ (for Algorithm~I) and $s$ (for Algorithm~II). As it is seen, Algorithm~I is more accurate only when $M$ is not too large.

\begin{figure}
\includegraphics[width=\textwidth]{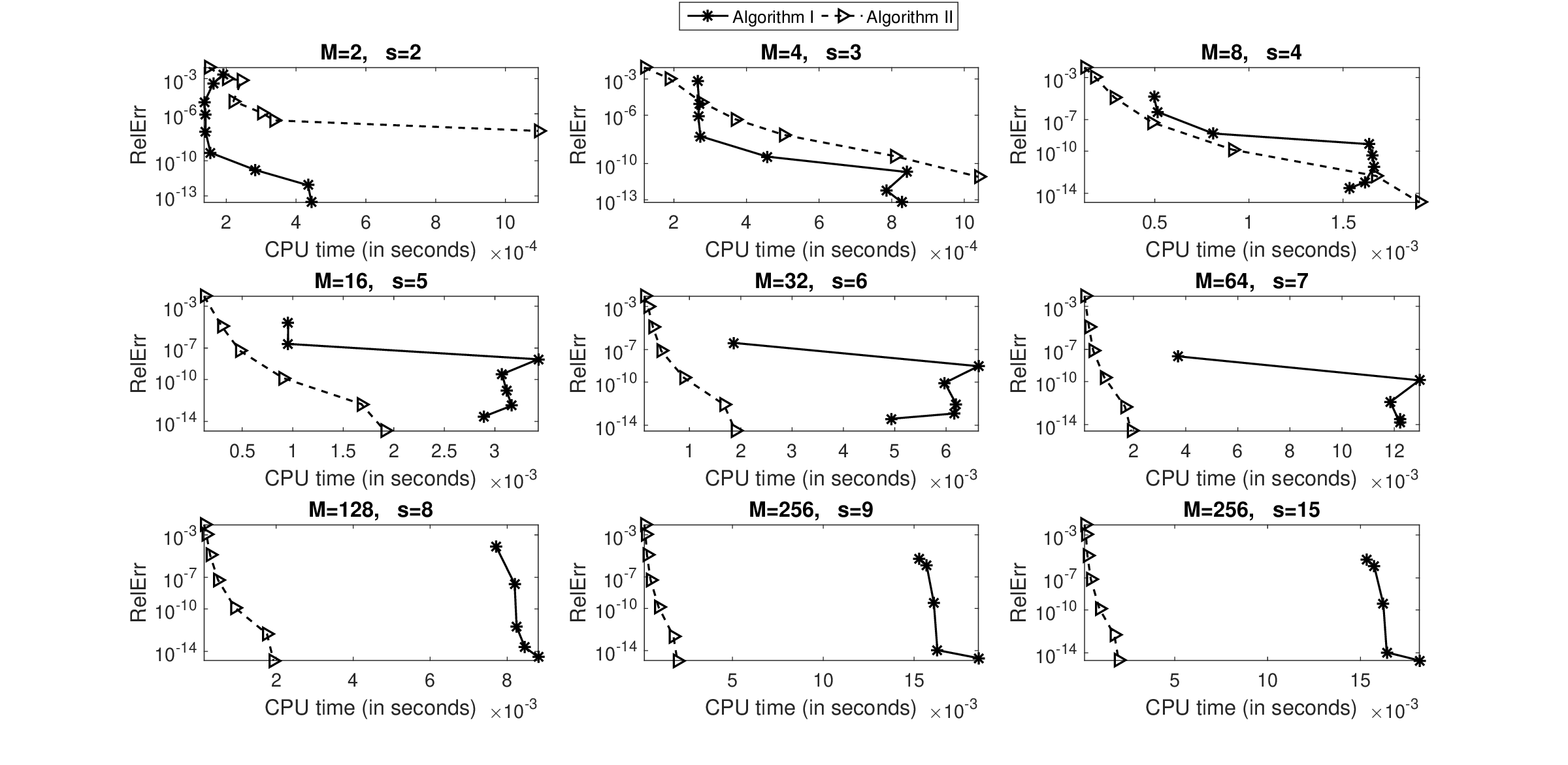}
\caption{\small Average CPU time (in seconds) vs relative error of Algorithms~I and~II when applied to the model integral~\eqref{eq:exmplcomp2}. The number $M$ of subintervals in Algorithm~I and the interpolation parameter $s$ in Algorithm~II are kept fixed in each plot while the accuracy increases by growing $N$. In Algorithm~II, ${N'=kN}$.\label{fig:comp2}}
\end{figure}
\paragraph{Experiment~2}
Here, we repeat Experiment~1 now for the model integral of~\S\ref{sub:numericI} with ${k=100}$, \ie
\begin{equation}\label{eq:exmplcomp2}
\int_{0}^1\frac{x^{4.5}}{1+x^2}\exp\left(100\,\im\sqrt{x^2+3x+4}\right)\,\df x.
\end{equation}
Note that the amplitude function is of the class $C^4[0,1]$. Thus, the error bound~\eqref{eq:err1} of Algorithm~I is valid when $N\le3$ and the error bound~\eqref{eq:err2} of Algorithm~II is valid when $s\le4$. Regardless to the error bounds, we employ both of the algorithms with larger $N$ and $s$ (Figure~\ref{fig:comp2}). Here again, as in Experiment~1, it is seen that Algorithm~I is more accurate only when $M$ is not too large.
\section{In the presence of stationary points}\label{sec:statpt}
Consider the integral~\eqref{eq:intfg}, and assume that $g'(x)$ vanishes at a finite number of points in $[a,b]$. It is said that a function $\varphi$ has a stationary point of order $n\ge1$ at $\xi\in[a,b]$ if
$$
\varphi'(\xi)=\cdots=\varphi^{(n)}(\xi)=0,\quad \varphi^{(n+1)}(\xi)\ne0.
$$

Assume that $g$ has a finite number of stationary points in $[a,b]$.  Then, we can always divide $[a,b]$ into subintervals such that $g$ has a single stationary point in each subinterval located at one of its endpoints. Thus, without loss of generality, we consider the integral ${I^{[0,1]}_k(f,g)}$, where $g$ has a single stationary point of order $n$ at $0$, and ${g'(x)>0}$ for any ${x\in(0,1]}$.

Under the change of variables $\tau=g(x)$, the integral is reduced to ${I^{[g(0),g(1)]}_{k}(F)}$, with $F$ defined by~\eqref{eq:F}. Again by the change of variables ${\tau=g(0)+2l\widehat{x}}$ with ${l=(g(1)-g(0))/2}$, the integral is transformed into one over the standard interval $[0,1]$:
\begin{align}
I^{[g(0),g(1)]}_{k}(F)=2l\exp(\im kg(0))I^{[0,1]}_{2lk}(\widehat{F}),
\end{align}
where
\begin{equation}
\widehat{F}(\widehat{x})=F(g(0)+2l\widehat{x}).
\end{equation}

By Theorem~4.1 of~\cite{dom13}, if $f$ and $g$ are smooth enough, ${\widehat{F}\in C_{\beta}^m[0,1]}$ for ${\beta=-n/(n+1)}$ and some $m$, depending on the degrees of smoothness of $f$ and $g$. Recall from~\cite{dom13} that for any ${\beta<0}$, ${C_{\beta}^m[0,1]}$ is defined as the space of all ${\varphi\in C(0,1]}$ such that
\begin{equation}
\|\varphi\|_{m,\beta}:=\max\left\{\sup_{x\in(0,1]}\left|x^{j-\beta}\varphi^{(j)}(x)\right|,\quad j=0,\ldots,m\right\}<\infty.
\end{equation}

Now the composite algorithm, introduced in~\cite{dom13}, can be employed for computing ${I^{[0,1]}_{2lk}(\widehat{F})}$. The algorithm employs the classical graded mesh
\begin{equation}\label{eq:gradmesh}
\Pi_{M,q}:=\left\{x_j:=\left(\frac{j}{M}\right)^q :\quad j=0,1,\ldots,M\right\},
\end{equation}
for some parameter ${q>1}$ sufficiently large. Then in each panel ${[x_{j-1},x_j]}$, the (${N+1}$)-point FCC rule is applied.

Our proposed algorithm here is almost the same. The only difference is that we employ the (${N+1}$)-point MFCC rule in each panel. In the following, we show how the (${N+1}$)-point MFCC rule may be employed, and how it affects the total error of the algorithm.
\subsection{The composite MFCC rules on graded meshes}
Divide the integration interval $[0,1]$ by the graded mesh $\Pi_{M,q}$ into $M$ panel and in each panel use the (${N+1}$)-point MFCC rule~\eqref{eq:mfcc} with ${N'=N}$ and ${u_j=t_{N-j,N}}$, \ie ${N+1}$ Clenshaw-Curtis points. As it is seen, the only difference of the algorithm here from Algorithm~I lies in the type of the partitioning ${0=x_0<\cdots<x_M=1}$: Uniform meshes for Algorithm~I and graded meshes for the current algorithm. The following theorem provides an error bound for the composite MFCC rules on the graded meshes.
\begin{theorem}\label{thm:staterr}
Let $f$ and $g$ be so smooth that $\widehat{F}\in C^{N+1}_{\beta}[0,1]$, and let $g$ has a single stationary point at $0$ of order $n\ge1$. Let ${0\le r< 1+\beta}$ with ${\beta=-n/(n+1)}$ and choose
\begin{equation}\label{eq:q}
q>(N+1-r)/(\beta+1-r).
\end{equation}
Then, the composite MFCC rule on the graded mesh~\eqref{eq:gradmesh} has the following error bound provided that $M$ is large enough:
\begin{equation}\label{eq:errgrad}
Ck^{-r}M^{-N-1+r}\|\widehat{F}\|_{N+1,\beta}+C'k^{-1}\sqrt{N}M^{-N}
\end{equation}
where $C$ and $C'$ are constants independent of $k$ and $M$.
\end{theorem}
\begin{proof}
Theorem~\ref{thm:generr} implies that the total error of a composite MFCC rule, is the sum of the error of the composite FCC rule and the sum over all the panels of the interpolation maximum errors multiplied by $C'k^{-1}\sqrt{N}$. The error of the composite FCC rule is estimated by Theorem~3.6 of~\cite{dom13} as
\begin{equation}\label{err:compfcc}
Ck^{-r}M^{-N-1+r}\|\widehat{F}\|_{N+1,\beta}.
\end{equation}

The largest panel of $\Pi_{M,q}$, i.e. the last one, is ${[(1-1/M)^q,1]}$, and it is rather small only if $M$ is large enough. For the second term of the error, the results in~\cite{ric69} or~\cite[\S8.3]{atk97}, and the inequalities~\eqref{eq:ineqleb} for small subintervals imply that the maximum interpolation error in each panel is of order ${\mathcal{O}\left(M^{-N-1}\right)}$. Note that, since ${\beta<0}$, the integration over the first panel ${[x_0,x_1]}$ is simply approximated by 0 (see Eq.~(3.5) in~\cite{dom13}), so the interpolation error in the first panel vanishes. Thus, the second term of the bound for the total error can be written as
\begin{equation}
C'k^{-1}\sqrt{N}(M-1)M^{-N-1},
\end{equation}
and this with~\eqref{err:compfcc} yields the result.
\end{proof}
\subsection{Numerical experiments}
We try to illustrate the error estimate~\eqref{eq:errgrad} by a set of numerical experiments. Consider the following sample integral
\begin{equation}\label{eq:sampint}
\int_{0}^1\frac{x-1}{1+x^2}\exp\left(\im k x^4\right)\,\df x,
\end{equation}
with some $k>0$. Clearly 0 is the only stationary point of the oscillator function, and it is of order 3. So, ${\beta=-3/4}$. Also, the amplitude and the oscillator functions are so smooth that $\widehat{F}\in C^{N+1}_{\beta}[0,1]$ for any $N>0$. Thus, Theorem~\ref{thm:staterr} is applicable.

\begin{table}
\renewcommand{\arraystretch}{1.4}
\caption{Application of the composite MFCC rules with the graded meshes $\Pi_{M,q}$ to the integral~\eqref{eq:sampint} with ${k=1000}$: Absolute error with its decaying rate.\label{tbl:grad}}
\begin{center} \footnotesize
\begin{tabular}{c*{8}{c}}
\cline{2-9}
& \multicolumn{2}{c}{$N=2$} & \multicolumn{2}{c}{$N=4$} & \multicolumn{2}{c}{$N=6$} & \multicolumn{2}{c}{$N=8$} \\
\cline{1-9}
$M$& error & rate & error & rate & error & rate & error & rate \\
\cline{1-9}
32  & $9.89\times10^{-3}$ &    --- & $3.27\times10^{-2}$ &    ---   & $2.95\times10^{-1}$ &    ---   & $1.50\times10^{2}$ &    ---  \\
64  & $5.99\times10^{-4}$ &    4.0 & $8.80\times10^{-5}$ &    8.5   & $3.02\times10^{-4}$ &    9.9   & $2.81\times10^{-3}$ &   15.7  \\
128 & $5.35\times10^{-5}$ &    3.5 & $8.25\times10^{-6}$ &    3.4   & $2.78\times10^{-6}$ &    6.8   & $1.06\times10^{-6}$ &   11.4  \\
256 & $4.77\times10^{-6}$ &    3.5 & $1.15\times10^{-7}$ &    6.2   & $1.16\times10^{-8}$ &    7.9   & $1.17\times10^{-9}$ &    9.8  \\
512 & $1.99\times10^{-6}$ &    1.3 & $6.45\times10^{-9}$ &    4.2   & $2.62\times10^{-11}$ &    8.8   & $6.05\times10^{-13}$ &   10.9  \\
\cline{1-9}
\end{tabular}
\end{center}
\end{table}
\paragraph{Experiment~1}
If $M$ increases while the other parameters are kept fixed, Theorem~\ref{thm:staterr} implies that the convergence rate is at least of order $\mathcal{O}(M^{-N+1})$. In order to illustrate this result, consider the sample integral~\eqref{eq:sampint} with $k=1000$. We employ the composite MFCC rule on the graded mesh~\eqref{eq:gradmesh} with ${q=\lfloor(N+1)/(\beta+1)\rfloor+1}$ for some $N$ and $M$. By $\lfloor x\rfloor$ we mean the largest integer not greater than $x$. Then, the condition~\eqref{eq:q} is satisfied. The absolute error with its decaying rate, as $M$ grows, is given in Table~\ref{tbl:grad}.

\begin{figure}
\centerline{\includegraphics[width=\textwidth]{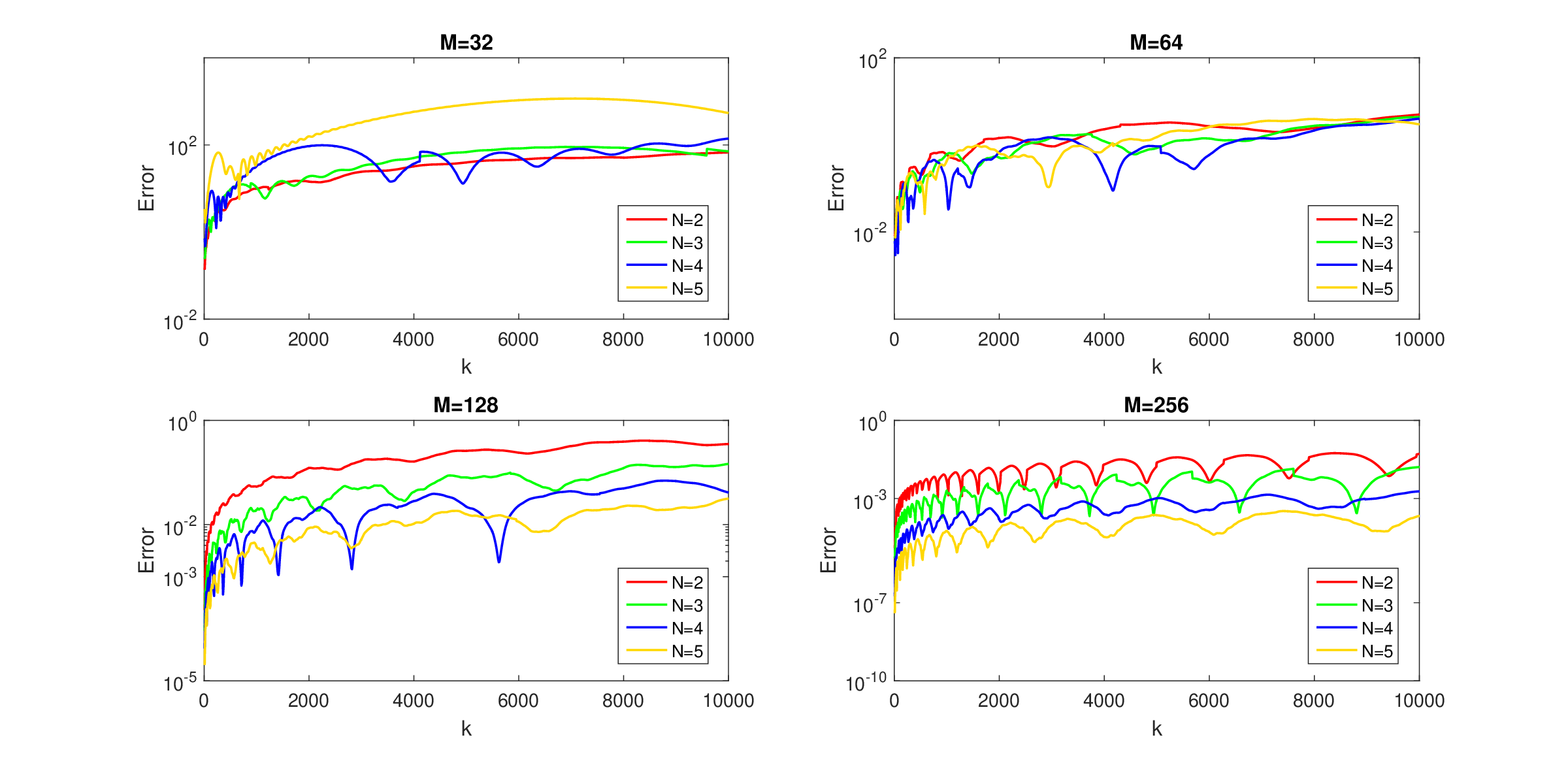}}
\caption{\small Absolute error scaled by $k$ as a function of $k$ for the composite MFCC rule on the graded mesh~\eqref{eq:gradmesh} with $q=\lfloor(N+1)/(\beta+1)\rfloor+1$ when applied on the model integral~\eqref{eq:sampint}. \label{fig:grad}}
\end{figure}
\paragraph{Experiment~2}
If $k\to\infty$ and other parameters remain fixed, the error estimate~\eqref{eq:errgrad} implies that the rule converges with the rate of $\mathcal{O}(k^{-1})$. In order to illustrate this rate, we again consider the sample integral~\eqref{eq:sampint} with $k$ varying in a wide band from 10 to $10^4$. For $M=32,64,128,256$ and $N=2,3,4,5$, we employ the composite MFCC rule with the graded mesh~\eqref{eq:gradmesh} and ${q=\lfloor(N+1)/(\beta+1)\rfloor+1}$. In Figure~\ref{fig:grad}, for each pair ${(M,N)}$, the absolute error scaled by $k$ has been plotted as a function of $k$.
\section{Conclusions}
We have introduced a general method, which can generate various modifications of the FCC rules and their composite versions. The modified rules can be applied to the oscillatory integral~\eqref{eq:intfg}, while they do not deal with the inverse of $g$. The main tool that allows us to design such modifications is interpolation; any kind of the modifications corresponds with a certain interpolation method.

When $g$ has no stationary points in $[a,b]$, two algorithms based on MFCC rules have been introduced. For each one, an error estimate has been given theoretically, and then illustrated by some numerical experiments. Overall, we found both the algorithms reliable and efficient for any regular integral~\eqref{eq:intfg} where $g'(x)\neq0$, $x\in[a,b]$.

When $g$ has a finite number of stationary points in $[a,b]$, it is possible to develop composite rules based on MFCC rules and graded meshes. Similar to the previous case, an error estimate has been given theoretically, and then illustrated by some numerical experiments. Such rules are reliable only for finely graded meshes. Especially, when grading parameter $q$ in $\Pi_{M,q}$ is rather large, the number $M$ of panels should be large enough. Thus, such rules are not recommended when $\beta$ is too close to $-1$ because in this case the panels in a narrow neighborhood of $0$ becomes too small. Then, evaluating functions at the interpolation points in such panels may result in large rounding error. This is the same effect which happens for the composite FCC rules (see~\S4.1 of~\cite{dom13}).
\bibliographystyle{siam}
\bibliography{ref}

\begin{thebibliography}{10}

\bibitem{abr65}
{\sc M.~Abramowitz and I.~Stegun}, {\em Handbook of mathematical functions with
  formulas, graphs, and mathematical tables}, Dover Publications, New York,
  1965.

\bibitem{atk97}
{\sc K.~Atkinson}, {\em The numerical solution of integral equations of the
  second kind}, Cambridge University Press, Cambridge, 1997.

\bibitem{ber04}
{\sc J.-P. Berrut and L.~Trefethen}, {\em Barycentric {L}agrange
  interpolation}, SIAM Rev., 46 (2004), pp.~501--517.

\bibitem{bru12}
{\sc O.~Bruno and D.~Hoch}, {\em Numerical differentiation of approximated
  functions with limited order-of-accuracy deterioration}, SIAM J. Numer.
  Anal., 50 (2012), pp.~1581--1603.

\bibitem{chan12}
{\sc S.~Chandler-Wilde, I.~Graham, S.~Langdon, and E.~Spence}, {\em
  Numerical-asymptotic boundary integral methods in high-frequency scattering},
  Acta Numerica, 21 (2012), pp.~89--305.

\bibitem{cle60}
{\sc C.~Clenshaw and A.~Curtis}, {\em A method for numerical integration on an
  automatic computer}, Numer. Math., 2 (1960), pp.~197--205.

\bibitem{dea09}
{\sc A.~Dea\~{n}o and D.~Huybrechs}, {\em Complex {G}aussian quadrature of
  oscillatory integrals}, Numer. Math., 112 (2009), pp.~197--219.

\bibitem{dom13}
{\sc V.~Dom\'{\i}nguez, I.~Graham, and T.~Kim}, {\em Filon-{C}lenshaw-{C}urtis
  rules for highly oscillatory integrals with algebraic singularities and
  stationary points}, SIAM J. Numer. Anal., 51 (2013), pp.~1542--1566.

\bibitem{dom11}
{\sc V.~Dom\'{\i}nguez, I.~Graham, and V.~Smyshlyaev}, {\em Stability and error
  estimates for {F}ilon-{C}lenshaw-{C}urtis rules for highly oscillatory
  integrals}, IMA~J. Numer. Anal., 31 (2011), pp.~1253--1280.

\bibitem{fil28}
{\sc L.~Filon}, {\em On a quadrature formula for trigonometric integrals},
  Proc.~Roy.~Soc.~Edinburgh, 49 (1928), pp.~38--47.

\bibitem{hig04}
{\sc N.~Higham}, {\em The numerical stability of barycentric {L}agrange
  interpolation}, IMA J. Numer. Anal., 24 (2004), pp.~547--556.

\bibitem{huy12}
{\sc D.~Huybrechs and S.~Olver}, {\em Superinterpolation in highly oscillatory
  quadrature}, Found. Comput. Math., 12 (2012), pp.~203--228.

\bibitem{huy06}
{\sc D.~Huybrechs and S.~Vandewalle}, {\em On the evaluation of highly
  oscillatory integrals by analytic continuation}, SIAM~J. Numer. Anal., 44
  (2006), pp.~1026--1048.

\bibitem{ise05}
{\sc A.~Iserles and S.~N{\o}rsett}, {\em Efficient quadrature of highly
  oscillatory integrals using derivatives}, Proc. R. Soc. Lond. A, 461 (2005),
  pp.~1383--1399.

\bibitem{ma18}
{\sc Y.~Ma and Y.~Xu}, {\em Computing highly oscillatory integrals}, Math.
  Comp., 87 (2018), pp.~309--345.

\bibitem{maj15}
{\sc H.~Majidian}, {\em Creating stable quadrature rules with preassigned
  points by interpolation}, CALCOLO, 63 (2016), pp.~217--226.

\bibitem{maj19}
\leavevmode\vrule height 2pt depth -1.6pt width 23pt, {\em Computing highly
  oscillatory integrals}, BIT, 59 (2019), pp.~155--181.

\bibitem{olv07}
{\sc S.~Olver}, {\em Moment-free numerical approximation of highly oscillatory
  integrals with stationary points}, Euro. J. Appl. Maths., 18 (2007),
  pp.~435--447.

\bibitem{olv08}
\leavevmode\vrule height 2pt depth -1.6pt width 23pt, {\em Numerical
  Approximation Of Highly Oscillatory Integrals}, PhD thesis, University of
  Cambridge, 2008.

\bibitem{pie00}
{\sc R.~Piessens}, {\em Computing integral transforms and solving integral
  equations using {C}hebyshev polynomial approximations}, J. Comput. Appl.
  Math., 121 (2000), pp.~113--124.

\bibitem{pie73}
{\sc R.~Piessens and M.~Branders}, {\em The evaluation and application of some
  modified moments}, BIT, 13 (1973), pp.~443--450.

\bibitem{pie76}
\leavevmode\vrule height 2pt depth -1.6pt width 23pt, {\em Numerical solution
  of integral equations of methematical physics, using {C}hebyshev
  polynomials}, J. Comput. Phys., 21 (1976), pp.~178--196.

\bibitem{pie84}
\leavevmode\vrule height 2pt depth -1.6pt width 23pt, {\em Computation of
  {F}ourier transform integrals using {C}hebyshev series expansions},
  Computing, 32 (1984), pp.~177--186.

\bibitem{pie71}
{\sc R.~Piessens and F.~Poleunis}, {\em A numerical method for the integration
  of oscillatory functions}, BIT, 11 (1971), pp.~317--327.

\bibitem{ric69}
{\sc J.~Rice}, {\em On the degree of convergence of nonlinear spline
  approximation}, in Approximations with special emphasis on spline functions,
  I.~J. Schoenberg, ed., 1969, pp.~349--365.

\bibitem{rut90}
{\sc H.~Rutishauser}, {\em Vorlesungen \"{u}ber numerische {M}athematik},
  vol.~1, Birkh\"{a}user, Basel, Stuttgart, 1976; English translation,
  \textit{Lectures on Numerical Mathematics}, Walter Gautschi, ed.
  Birkh\"{a}user, Boston, 1990.

\bibitem{slo78}
{\sc I.~Sloan and W.~Smith}, {\em Product integration with the
  {C}lenshaw-{C}urtis and related points: {C}onvergence properties}, Numer.
  Math., 30 (1978), pp.~415--428.

\bibitem{slo80}
\leavevmode\vrule height 2pt depth -1.6pt width 23pt, {\em Product integration
  with the {C}lenshaw-{C}urtis points: {I}mplementation and error estimates},
  Numer. Math., 34 (1980), pp.~387--401.

\bibitem{xia07}
{\sc S.~Xiang}, {\em Efficient {F}ilon-type methods for
  {$\int^b_af(x)e^{\im\omega g(x)}\,\df x$}}, Numer. Math., 105 (2007),
  pp.~633--658.

\end{thebibliography}
\end{document}